\documentclass[12pt]{amsart}%
\usepackage{amsmath,amsthm,amsfonts,amssymb,eepic}
\usepackage%[hypertex]
{hyperref}
\usepackage{url} 
\usepackage{multicol}
\usepackage[margin=1.4in]{geometry}

\theoremstyle{plain}

\newtheorem{corollary}{Corollary}%[section]
\newtheorem{lemma}{Lemma}%[section]
\newtheorem{remark}{Remark}%[section]
\newtheorem{theorem}{Theorem}%[section]
\newtheorem{conjecture}{Conjecture}%[section]
\newcommand{\F}[1]{\ensuremath{\mathfrak{#1}}}
\newcommand{\C}[1]{\ensuremath{\mathcal{#1}}}
\newcommand{\B}[1]{\ensuremath{\mathbb{#1}}}

\date{\today}

\begin{document}

\title[Meromorphic Spectral Zeta Functions on Fractals]{Existence of a Meromorphic Extension of Spectral Zeta Functions on Fractals}

\author{Benjamin A.~Steinhurst}
\address{Department of Mathematics and Computer Science, McDaniel College, Westminster, MD, 21157
%\address{Department of Mathematics, Cornell University, Ithaca, NY 14853-4201 \url{http://www.math.cornell.edu/~steinhurst/}}%, USA
\url{http://www2.mcdaniel.edu/bsteinhurst/}}
\email{{bsteinhurst@mcdaniel.edu}}
\thanks{Research supported in part by the National Science Foundation, grants DMS-0652440 (first author) and DMS-0505622 (second author).}
%The authors are supported by the NSF.}\footnote{Research supported in part by the National Science Foundation, grants DMS-0652440 (first author) and DMS-0505622 (second author).}

\author{Alexander~Teplyaev}
%\thanks{}\thanks{Research of both authors is supported in part by the National Science Foundation.}

\address{Department of Mathematics, University of Connecticut, Storrs, CT, 06269 \url{http://www.math.uconn.edu/~teplyaev/}}%, USA
\email{{teplyaev@uconn.edu}}

%\subjclass[2000]{Primary 28A80; Secondary 35P05, 35J05}
%\keywords{Dirichlet forms, fractals, diffusions}
%\thanks{This paper is in final form and no version of it will be submitted for publication elsewhere.}

\begin{abstract}
We investigate the existence of the meromorphic extension of the spectral zeta function of a Laplacian on self-similar fractals using the results of Kigami and Lapidus (based on renewal theory) and the newer results by Hambly and Kajino based on heat kernel estimates and other probabilistic techniques. We also formulate conjectures which hold true for the examples that have been analyzed in the existing literature. 
\end{abstract}


\maketitle

There have been many  works in mathematical physics, analysis, and probability on fractals  studying spectral and heat kernel asymptotics of various Laplacians on fractal sets, see  \cite[and references therein]{ADT,ADT2,eigenpapers,eigenpaper2,DerfelEtAl2008,DGV,Hambly2010,Kajino2010,LapidusvF2006,Teplyaev2007}. It is possible that fractal spaces may provide useful models for the study of quantum gravity \cite{AJL-univ}. In particular, on many fractals the short time asymptotics of the partition function are not given by a power function alone as it is on manifolds but in many cases the power function is corrected by a multiplicatively periodic function. This behavior has been observed  in \cite{KigamiLapidus1993,GrabnerWoess1997,Grabner1997,Str2009} for finitely ramified fractals, and \cite{ADT,Hambly2010,Kajino2010} extend the class of fractals for which one can expect the log-periodic oscillations in the short time heat kernel asymptotics.

In this paper we investigate the related question of the existence of a meromorphic continuation of the spectral zeta function, which has found many profitable uses in physics \cite{Elizalde1995,Kirsten2010} (e.g. Casimir effect \cite{BCOR,Casimir1997}). If the Weyl ratio for the eigenvalue counting function is a multiplicatively periodic function, up to a smaller order term (as proved by Kigami and Lapidus in \cite{KigamiLapidus1993}),  then the spectral zeta function can be expected to be meromorphic in some region to the left of $d_S$ where $d_S$ is the spectral dimension of the underlying Laplacian. We discuss how  new results by Hambly and Kajino \cite{Hambly2010,Kajino2010,KajinoAMS}  can be applied to obtain a meromorphic continuation of the spectral zeta function of the Laplacian on certain fractals, such as finitely ramified symmetric fractals and Sierpinski carpets. Furthermore, if the partition function is decomposed into a sum of power functions times multiplicatively periodic terms, and an exponentially decreasing term (with no other terms), then the spectral zeta function is meromorphic over the whole complex plane (this is done, in relation to \cite{ADT} concerning the physical implications the existence of meromorphic continuations).

If the Laplace operator $L$ has a discrete spectrum with eigenvalues $\lambda_l$, 
repeated according to their multiplicities, then the spectral zeta function of $L$ is given by 
\begin{equation}
	\zeta(s,\gamma) = \sum_{l=1}^{\infty} (\lambda_l + \gamma)^{-s/2}
\end{equation} 
whenever the series converges absolutely. The use of $s/2$ instead of $s$ is not essential, but is precedented in the cited literature and only changes the results by a scaling factor of $2$. 
Recall that the partition function of  a non-negative self-adjoint operator $L$ is  
$
	Z_L(t) = Tr(e^{-tL})$, which decays exponentially for large $t$ in the case of a discrete spectrum with no or excluded zero eigenvalue. 
By applying the inverse Mellin transform (\cite{DaviesBook,Debnath2007})  to $Z_L(t)$, we have
\begin{equation}
	\zeta(2s,\gamma) = \frac{1}{\Gamma(s)} \int_0^{\infty} t^{s-1}Z_L(t)e^{-\gamma t}\ dt.
\end{equation}
%\begin{definition}\label{def1} In this paper w
We consider Laplacians on self-similar compact sets $F$, which are connected metric space with injective contraction maps $\{\psi_j\}_{j=1}^N$  such that $\psi_j:F\to F$ and $F = \bigcup_{j=1}^{N} \psi_j(F)$. For the sake of simplicity, we only consider the unique probability self-similar measure $\mu$ on $F$ with equal weights, that is $\mu(\psi_j(F))=1/N$. On a self-similar set, in addition to a self-similar metric and measure, one could ask what it means for a Dirichlet form to be self-similar. We assume the existence of the following decomposition of a local regular Dirichlet form $\mathcal{E}$ on $F$
\begin{equation}\label{e-s-s-r-f}
	\mathcal{E}(f,g) = \rho_F \sum_{i=1}^{N} \mathcal{E}(f\circ \psi_i,g\circ \psi_i),
\end{equation}
and note that the effect of applying $\mathcal{E}$ to $f \circ \psi_i$ is to localize $\mathcal{E}$ to act only on $\psi_i(F)$. 
With appropriate boundary conditions and domain, the Laplacian $\Delta$ is defined in a weak sense by $\mathcal{E}(f,g) = -\int_F f\Delta g\,d\mu$.

%\begin{remark} 
The constant $\rho_F$ is called the energy rescaling factor or conductance scaling factor. Its reciprocal $r=\frac1{\rho_F}$ is the resistance scaling factor. In a number of papers  the value of $\rho_F$ is explicitly calculated for various finitely ramified fractals (see \cite{eigenpapers,eigenpaper2,BFPRST,HMT,Ki,Ki2,St06book,Tcjm} and references therein), however for infinitely ramified fractals the only examples where the exact energy scaling was obtained are \cite{BDMS,Steinhurst,S-POTA} (information on  eigenvalues and eigenfunctions can be found in \cite{KKPSS,RomeoSteinhurst}). For generalized Sierpinski carpets,  $\rho_F$ can be estimated and the uniqueness can be proved (see \cite{BB,BB1,BB2,BB3,HKKZ,BBKT} and references therein). 
Note that the spectral dimension is given by $$d_S=\frac{2\log(N)}{\log(\tau)}=\frac{2\log(N)}{\log(\rho_FN)}=2\frac{d_f}{d_w}=2\frac{d_0}{d_w}$$
where $d_f=d_0$ is the Hausdorff dimension and $d_w$ is the so-called walk dimension  (see Lemma~\ref{ExpBounds} and the related work \cite{St03jfa} by Strichartz on the spectral dimension). The Laplacian scaling factor $\tau=\rho_F N$ is also known as the time scaling factor. 
%, is related to the spacing of poles described below via the parameter $\log(\rho_F N)$. 
%\end{remark}

We say that a self-similar set $K$ is finitely ramified and 
fully symmetric if 
 the following three conditions hold: 
 \begin{enumerate}
 	\item there exists a finite subset $V_0$  of $K$ such that $\psi_j(K) \cap \psi_k(K) = \psi_j(V_0) \cap \psi_k(V_0)$ for $j \neq k$ (this  intersection may be empty); 
 	\item if $v_0\in V_0\cap\psi_j(K)$ then $v_0$ is the fixed point of $\psi_j$;
 	\item  there is a group $\F G$ of isometries of $K$ that has a doubly transitive action on $V_0$ and is compatible with the self-similar structure $\{\psi_j\}_{j=1}^N$, which means (\cite[Proposition 4.9]{NT} and also \cite{eigenpapers,HSTZ,MT}) that for any $j$ and any $g\in\F G$ there exists $k$ such that $g\circ\psi_j=\psi_k$.
 \end{enumerate}
 Moreover, a  fully symmetric finitely ramified self-similar set $K$ is a  post-critically finite (p.c.f.) self-similar set if and only if for any $v_0\in V_0$ there is a unique $j$ such that $v_0\in\psi_j(K)$ \cite{HSTZ,NT}. 
%\end{definition} 

Our first  result is the following theorem, which improves the main result in  \cite[Theorem 2]{Teplyaev2007} (this  result is related to the gaps in the spectrum, see \cite{HSTZ}) which states that for the spectral zeta function for a polynomial with a totally disconnected Julia set then the zeta function extends to the left of the tower of poles at $d_S$ by some $\epsilon$. The connection with spectral zeta functions in a fractal context is given through \cite[Theorem 7]{Teplyaev2007} in the example of the Sierpinski gasket. The spectrum of the Laplacian on the Sierpinski gasket if given via the Julia set of a polynomial using the spectral decimation method. 

\begin{theorem}
On any  fully symmetric p.c.f. fractal, as defined above, the spectral zeta function with $\gamma = 0$  has a meromorphic continuation to $Re(s)>-\epsilon$  for some positive $\epsilon$ with at most two sequences  of poles, also called spectral dimensions, at $Re(s)=d_S$ and $Re(s)=0$.
\end{theorem}
\long\def\BBF#1#2#3{\begin{figure}[htb]#3\caption{#2}\label{#1}\end{figure}} 
\BBF{figSigCSD}
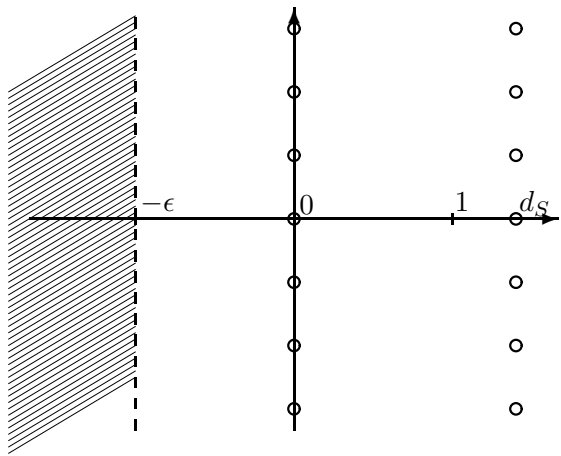
{Complex spectral dimensions of the Laplacian on a fully symmetric finitely ramified fractal. %Unpublished 
New results of Kajino \cite{KajinoAMS} imply that in fact there is a meromorphic continuation for all $s\in\mathbb C$.}
{\begin{picture}(250,170)(-125,-85) 
\setlength{\unitlength}{0.5pt}\small
\thicklines\setlength{\unitlength}{0.4pt}
\put(-250,0){\vector(1,0){500}}
\put(0,-200){\vector(0,1){400}}
\put(5,5){{$0$}}
%\put(103,7){$\sm{d_R}$}
\put(152,7){${1}$}
\put(-145,7){${-\epsilon}$}
\put(150,-5){\line(0,1){10}}
\put(-150,-5){\line(0,1){10}}
\put(213,7){${d_S}$}
%\multiput(100,-180)(0,60){7}{\circle*{8}}
\multiput(210,-180)(0,60){7}{\circle{11}}
\multiput(0,-180)(0,60){7}{\circle{11}}
\multiput(-150,-200)(0,20){20}{\line(0,1){10}}
\thinlines
\multiput(-150,-150)(0,6){58}{\line(-5,-3){120}}
\end{picture}} 
\begin{proof}It is easy to see that the spectral zeta function is analytic for $Re(s)>d_S$ and there is a simple pole at $s=d_S$. From the results  in \cite{eigenpapers,Teplyaev2007} we obtain that there exists a meromorphic continuation to the half-plane $Re(s)>-\epsilon$  with finitely many sequences of poles in $0\leqslant Re(s)\leqslant d_S$. In addition, \cite[Theorem 7.7 and Corollary 7.8]{Kajino2010} and Lemma~\ref{lem-L} applied with $\gamma = 0$ imply that there are no poles in $0< Re(s)< d_S$. Note that according to \cite[Definition 2.10 and Definition 6.8]{Kajino2010}, a fully symmetric p.c.f. fractal has zero-dimensional rational boundary and so there are heat kernel estimates (see \cite{KiHKE}). Thus the only possible sequences of poles are at $Re(s)=d_S$ and $Re(s)=0$, which completes the proof. 
\end{proof}

For ease of reference, \cite[Theorem 7.7 and Corollary 7.8]{Kajino2010} are reproduced at the end o this paper. See Theorem \ref{thm:kajino} and Corollary \ref{cor:kajino}.

\begin{theorem}\label{thm:funcequal}For any intersection type finite self-similar structures (see \cite[Theorem 7.7 and Corollary 7.8]{Kajino2010}), including fully symmetric p.c.f. fractals, nested fractals and 
generalized Sierpinski carpets,  the spectral zeta function associated to the  self-similar Laplacian has a meromorphic extension to beyond the spectral dimension, at least to the  half-plane $Re(s)>2\frac{d_\partial}{d_w}$. 
Moreover, the spectral zeta function satisfies the following functional equation for $\gamma > -c_4$ and $Re(s) > 2\frac{d_{\partial}}{d_w}$
 \begin{equation}\label{l-dif}
 	\frac{d}{d\gamma}\zeta(s,\gamma) = -\gamma \zeta(s+2,\gamma).
\end{equation}
The poles of $\zeta(s,0)$ are located, in the region $Re(s)>2\frac{d_\partial}{d_w}$, at \ $d_S + \frac{4\pi i n}{\log( \tau  )}$. When $\gamma \neq 0$ they are located at \ $d_S - 2m + \frac{4\pi i n}{\log( \tau  )}$ for $m \ge 0$.
\end{theorem}

\begin{proof}Similarly to the previous result, this is implied by Lemma~\ref{lem-L} and \cite{Hambly2010} and \cite[Theorem 7.7 and Corollary 7.8]{Kajino2010} (See Theorem \ref{thm:kajino} and Corollary \ref{cor:kajino}.)
The same argument, as in the proof of Lemma~\ref{lem-L}, for differentiating under the integral for $I_1(s,\gamma)$ applies also to $I_2(s,\gamma)$ and $I_3(s,\gamma)$ with the same functional equation. The location of the poles when $\gamma = 0$ is observed from directly summing the series in (\ref{eq:I1series}) when $\gamma = 0$. When $\gamma \neq 0$ the poles when $m=0$ are obtained from the  same estimate, and the translations of poles by $2m$ is forced by the functional equation. 
\end{proof}

\begin{remark} In the case of the standard Sierpinski carpet $2\frac{d_f}{d_w} - 2$ will be less than $2\frac{d_{\partial}}{d_w}$ so that there are no extra poles in the right half-plane.
In fact, the spectral zeta function associated to the self-similar Laplacian on the Sierpinski carpet has a meromorphic extension to the whole complex plane  because, by the work \cite
%[Theorem 4]
{KajinoAMS} of Kajino, the conditions of Theorem \ref{ExpBounds} are satisfied. Moreover, the same is true for large classes of fractals, such as nested fractals and generalized Sierpinski carpets where the values of $d_k$ have a geometric meaning. For example $d_0 = d_f$ and $d_k$ is the Minkowski dimension of the co-dimension $k$ faces of the carpet and $d_d = 0$ is the Minkowski dimension of the single point that is a co-dimension $d$ face of the carpet. 
\end{remark}

\begin{conjecture}
For fully symmetric finitely ramified fractals, even without heat kernel estimates,  the spectral zeta function with $\gamma = 0$  has a meromorphic continuation to $\mathbb C$ with at most two sequences  of poles, also called spectral dimensions, at $Re(s)=d_S$ and $Re(s)=0$. This applies for the usual Dirichlet Laplacian, and for the Neumann Laplacian if the zero eigenvalue is excluded.  
\end{conjecture}

\begin{conjecture}
For  generalized Sierpinski carpets the possible poles of the spectral zeta function with real part $2\frac{d_k}{d_w}$,  with $k=1,\ldots, d-1$, are actually removable singularities  because there are different self-similar (graph-directed) structures that yield the same Laplacian operator. This applies for the usual   Neumann Laplacian if the zero eigenvalue is excluded. For the Dirichlet Laplacian the dimension of the boundary will play a role in the spectral asymptotics.  
\end{conjecture}

The two dimensional standard Sierpinski carpet can be realized by two such structures and in this case it is conjectured that there are only two sequences of poles one at $Re(s) = 2\frac{d_f}{d_w}$ and the other at $Re(s) = 0$. This has been observed in the case of some Laakso spaces in \cite{BDMS}. 

\begin{lemma}\label{lem-L}
Suppose that $d_{\partial} < d_f$ and for $t <1$
\begin{equation}
	c_1t^{-d_{\partial}/d_w} \le t^{-d_f/d_w}G\left( \log \frac{1}{t} \right) - Z_L(t) \le c_2t^{-d_{\partial}/d_w}
\end{equation}
where $G$ is a periodic function bounded above and away from zero with period $\log( \tau  )>0$, while for $t \ge 1$ there exist $c_3,c_4 \ge 0$ such that
\begin{equation}
	 |Z_L(t)| \le c_3e^{-c_4t}.
\end{equation}
Then, for any $\gamma > -c_4$, $\zeta(s,\gamma)$ has a meromorphic continuation for $Re(s)>2\frac{d_{\partial}}{d_w}$.
\end{lemma}

\begin{proof}
Note that inverse Mellin transformations are linear so that we may transform each of the asymptotics separately. By assumption there exist bounded measurable functions $B(t)$ and $C(t)$ such that for $t<1$
\begin{equation}
	Z_L(t) =  t^{-d_f/d_w}G\left(  \log \frac{1}{t} \right) + B(t)t^{-d_{\partial}/d_w}
\end{equation}
and for $t \ge 1$
\begin{equation}
	Z_L(t) = C(t)e^{-c_4t}.
\end{equation}
Then the Mellin transform of $Z_L(t)e^{-\gamma t}$ is 
{\small
\begin{align}
	\zeta(2s,\gamma)=&\  \frac{1}{\Gamma(s)}\int_0^{1} t^{s-1}t^{-d_f/d_w}G\left(  \log \frac{1}{t} \right) e^{-\gamma t} B(t)\ dt&\\
	 &+ \frac{1}{\Gamma(s)}\int_0^{1} t^{s-1}B(t)t^{-d_{\partial}/d_w} e^{-\gamma t}\ dt&\\
	 &+ \frac{1}{\Gamma(s)} \int_1^{\infty} t^{s-1}C(t)e^{-c_4t}e^{-\gamma t}\ dt \\
	  = &\   
	  I_1(s,\gamma) + 
	  I_2(s, \gamma) + 
	  I_3(s,\gamma).
\end{align}
}
Since $B(t)$ and $C(t)$ are bounded functions, they do not contribute to the divergence or convergence of these integrals and may be ignored without loss of generality. Note that for all $\gamma \in \B{R}$, $I_1(s,\gamma)$ converges for $Re(s) > \frac{d_f}{d_w}$ and $I_2(s,\gamma)$ converges for $Re(s) > \frac{d_{\partial}}{d_w}$, while $I_3(s,\gamma)$ converges for all $s \in \B{C}$ and $\gamma > -c_4$. It suffices to show that $I_1(s,\gamma)$ can be meromorphically extended to $Re(s) > \frac{d_{\partial}}{d_w}$.

Let $\log( \tau  )$ be the period of $G(T)$ so that $G(\log( \tau  ) T)$ has period 1 in the variable $T$. Recall that $\tau$ is the time scaling factor. Using the change of variables $t \mapsto  \tau^{  T}$ then
\begin{align}
	I_1(s,\gamma) &= \frac{\log( \tau  )}{\Gamma(s)} \int_{-\infty}^{0} ( \tau^{  T})^{s-d_f/d_w}e^{-\gamma  \tau^{  T}}G(\log( \tau  ) T)\ dT \label{eq:l1series1}\\
	&= \frac{\log( \tau  )}{\Gamma(s)} \sum_{p=-\infty}^{-1} \int_p^{p+1} \tau^{  T(s-d_f/d_w)}e^{-\gamma  \tau^{  T}}G(\log( \tau  ) T)\ dT. \label{eq:I1series}
\end{align}
The issue of convergence is only at $T= -\infty$ and thus the integral $I_1(s,\gamma)$ will converge if the summation converges absolutely. 
This can be established by using the Taylor series in $\gamma$. Moreover if the integral  $I_1(s,\gamma)$ converges for a specific pair $(s,\gamma)$ it will be analytic in $s$ in some small neighborhood of $s$ for that value of $\gamma$. 
Note that if  $s = x+iy$ with $x> \frac{d_f}{d_w}$, then
$$
	|I_1(s,\gamma)| 
	%& \le   \frac{\log( \tau  )}{|\Gamma(s)|} \sum_{p=-\infty}^{-1} \tau^{  (p+1)(x-d_f/d_w)}e^{-\gamma  \tau^{  p}} \int_p^{p+1}G(\log( \tau  ) T)\ dT\\
	%&= \frac{\log( \tau  )}{|\Gamma(s)|} G^{*} \sum_{p=-\infty}^{-1} \tau^{  (p+1)(x-d_f/d_w)}e^{-\gamma  \tau^{  p}}\\
	%& 
	\le  \frac{\log( \tau  )}{|\Gamma(s)|} G^{*} \max\{ 1, e^{-\gamma} \}\sum_{p=-\infty}^{-1} \tau^{  (p+1)(x-d_f/d_w)}
%\end{align}
$$
where $G^{*} = \int_p^{p+1} G(\log( \tau  ) T)\ dt$ which is independent of $p$ by the periodicity of $G(\log( \tau  ) T)$. This last sum is geometric in $p$ so if $x$ is replaced by $s = x+iy$ this bound on $|I_1(s,\gamma)|$ has a meromorphic extension to the complex plane with poles at $s = \frac{d_f}{d_w} + \frac{2i \pi n}{\log( \tau  )}$ for all fixed $\gamma \in \B{R}$.

The integrand in (\ref{eq:l1series1}) is for $\gamma > -c_4$ smooth in $\gamma$ and bounded by a $T-$integrable function independent of $\gamma$ in the region $Re(s)> \frac{d_f}{d_w}$. 
One can then take the derivative with respect to $\gamma$ inside the integral and obtain \eqref{l-dif}. 
%\begin{align*}
%	\frac{d}{d\gamma} I_1(s,\gamma) &= \frac{d}{d\gamma}\frac{\log( \tau  )}{\Gamma(s)} \int_{-\infty}^{0} ( \tau^{  T})^{s-d_f/d_w}e^{-\gamma  \tau^{  T}}G(\log( \tau  ) T)\ dT&\\
%	&=\frac{\log( \tau  )}{\Gamma(s)} \int_{-\infty}^{0} ( \tau^{  T})^{s-d_f/d_w} \frac{d}{d\gamma}e^{-\gamma  \tau^{  T}}G(\log( \tau  ) T)\ dT&\\
%	&= \frac{\log( \tau  )}{\Gamma(s)} \int_{-\infty}^{0} ( \tau^{  T})^{s-d_f/d_w}e^{-\gamma  \tau^{  T}} (-\gamma) \tau^{  T}G(\log( \tau  ) T)\ dT&\\
%	&= \frac{-\gamma \log( \tau  )}{\Gamma(s)} \int_{-\infty}^{0} ( \tau^{  T})^{(s+1)-d_f/d_w}e^{-\gamma  \tau^{  T}}G(\log( \tau  ) T)\ dT
%	%\\
%	= -\gamma I_1(s+1,\gamma).&
%\end{align*}
Repeating this argument it is possible to find $\frac{d^{l}}{d\gamma^{l}} I_1(s,\gamma)$ iteratively for $\gamma > -c_4$ and $Re(s) > \frac{d_f}{d_w}-l$. Since $I_1(s,\gamma)$ analytic in a right half-plane this implies that it varies smoothly in $\gamma \in (-c_4,\infty)$ and $I_1(s,\gamma)$ can be recovered by integrating $\frac{d^{l}}{d\gamma^{l}} I_1(s,\gamma)$ over $\gamma$. This not only gives a meromorphic extension of $I_1(s,\gamma)$ to $Re(s)> \frac{d_{\partial}}{d_w}$ but also to the whole complex plane (see next lemma for the use of this fact). Notice that if $Re(s) > \frac{d_f}{d_w} - l$ this definition is consistent with the definition of $I_1(s,\gamma)$ for $Re(s) > \frac{d_f}{d_w}$.
\end{proof}

\begin{lemma}\label{ExpBounds}
Suppose that for $t < 1$
\begin{equation}
	Z_L(t) =  \sum_{k=0}^{d} t^{-\frac{d_k}{d_w}}G_k\left(  \log \frac{1}{t} \right)  + O\left(
	\exp{(-ct^{-\frac1{d_w-1}})}\right)
\end{equation}
where the $G_k$ are periodic functions bounded above, and for $t\ge 1$ there exist $c_5,c_6 \ge 0$ such that
\begin{equation}
	-c_5e^{-c_6t} \le Z_L(t) \le c_3e^{-c_6t}.
\end{equation}
Then $\zeta(s,\gamma)$ has a meromorphic continuation to the complex plane for $\gamma > -c_6$. 
\end{lemma}

\begin{proof}
The technique for handling the $I_1$ term in Lemma \ref{lem-L} is repeated for each of the $t^{-\frac{d_k}{d_w}}G_k\left(  \log \frac{1}{t} \right)$ terms with their respective periods, giving their Mellin transforms and meromorphic continuations. Each of these meromorphic function have poles at $s = \frac{d_k}{d_w}  - 2m + \frac{2i\pi n}{\log( \tau  )}$ for $n \in \mathbb{Z}$ and $m \ge 0$. There is no analogue of the $I_2(s,\gamma)$ term. The $I_3(s,\gamma)$ term of Lemma \ref{lem-L} is now replaced with a term of the form
\begin{equation}
	\frac{1}{\Gamma(s)} \int_0^{1} t^{s-1}e^{(-ct^{-\frac1{d_w-1}})} e^{-\gamma t}\ dt,
\end{equation}
which converges if 
%\begin{equation}
%	\frac{1}{\Gamma(s)} \int_0^{\infty} t^{s-1}e^{(-ct^{-\frac1{d_w-1}})} e^{-\gamma t}\ dt
%\end{equation}
the same integral from $0$ to $\infty$ converges. It is known that the inverse Mellin transform of such an exponential term is the product of a complex exponential with base $c$ and a shifted Gamma function which is meromorphically extendable to the whole plane with well known poles in the left half-plane. The existence of the meromorphic extension of the integral of $t^{s-1}e^{-c_6t-\gamma t}$ is standard and also  
%\begin{equation}
%	\frac{1}{\Gamma(s)} \int_0^{\infty} t^{s-1}e^{-c_6t} e^{-\gamma t}\ dt
%\end{equation}
is precisely the argument of Lemma \ref{lem-L} concerning the $I_3$ term. The sum of these meromorphic function has discrete poles in a finite number of towers that do not accumulate thus $\zeta(s,\gamma)$ is meromorphic with complex dimensions whose real parts are given by $\frac{d_k}{d_w}$, where $d_0 = d_f$ and $G_0$ is not identically zero. Thus a meromorphic extension of $\zeta(2s,\gamma)$ can be given for the complex plane provided that $\gamma > -c_6$.
\end{proof}

\begin{corollary}
Under the 
assumptions 
of Lemma \ref{ExpBounds}, 
the functional equation of Theorem~\ref{thm:funcequal} holds in the whole complex plane. 
\end{corollary}

\section*{Appendix: cited Results by N.~Kajino from \cite{Kajino2010}}

In our work we use 
\cite[Theorem 7.7]{Kajino2010} and \cite[Corollary 7.8]{Kajino2010}
 by 
N. Kajino, which we cite in Theorem~\ref{thm:kajino} and Corollary~\ref{cor:kajino} 
in the end of this Appendix section. 
Although it is not possible to provide all the background for these important and interesting results, we make a somewhat abbreviated exposition  in this section. The reader is encourage to 
look in \cite{Kajino2010} for more details.

Unless mentioned otherwise, 
 in this section 
 we use notation, 
  definitions and results from \cite{Kajino2010}. 
%
%references made%, not to earlier statements in this paper. 
Let $K$ be the attractor of the iterated functions system $\{ F_i\}_{i \in S}$ for a finite index set $S$, where $F_i$ are contraction similarities in $\mathbb{R}^{d}$. Then $K$ is a compact subset of $\mathbb{R}^{d}$, and   $V_0$ denotes the boundary between $K$ and the unbounded component of its complement. 
Furthermore, let 
$\C{L} = ((K,S,\{F_i\}_{i \in S}),\mu,\C{E},\C{F},r=(r_i)_{i \in S})$ be a self-similar Dirichlet space, where 
$\mu$ is a self-similar measure with weights $(\mu_i)_{i \in S}$, and 
$\C{E}$ is a self-similar Dirichlet form on $K$ with domain $\C{F}$ and resistance weights $(\mu_i)_{i \in S}$. Note that in our work $\mu_i $, $r_i =\frac1{\rho_F}$, $\gamma_i $ are constant, which means they do not depend on $i$, and the self-similarity relation for 
Dirichlet form $\C{E}$ is the relation 
\eqref{e-s-s-r-f}.

Denote by $W_*$ the set of finite length words in the symbols $S$, 
by $W_{\#}$ the set of non-empty words of finite length, 
and by $\Sigma$ the collection of words of infinite length. 
We use $|w|$ to denote the length of a word in $W_*$. 
Then there is a   self-similar scale
$\C{S}=\{\Lambda_s\}_{s \in (0,1]}$, defined below, 
with weights $\gamma = (\gamma_i)_{i \in S}$, $\gamma_i := \sqrt{r_i \mu_i}$. 
We first define 
the gauge function $g(w) = \sqrt{r^{|w|} \mu(K_w)} = \prod_{i=1}^{|w|} \gamma_{w_i}$, 
and then define  
\begin{equation*}
	\Lambda_s := \{w_1\cdots w_m \in W_* | g(w_1 \cdots w_{m-1}) > s \ge g(w_1\cdots w_m)\}
\end{equation*} 
with the convention that $g(w_1w_2\ldots w_{m-1}) = 2$ when $m=0$. 
The collection $\C{S} = \{ \Lambda_s \}_{s \in (0,1]}$ is called the scale associated to the Dirichlet space. This scale defines the spectral dimension $d_S:= d(\gamma) =dim_{\C{S}} K > 0$ 
by the usual self-similarity relation $\sum_{i \in S}\gamma_i^{d_S}=1$. 
If $v \in \Lambda_s$, then $K_v = v\Sigma$, that is all infinite words beginning with $v$, can be roughly thought of as a set of ``radius'' $s$. 
When there is a metric which gives an equivalent topology we say the metric is adapted to the scale. 

Below we present a brief overview of the necessary conditions for a Dirichlet space to satisfy the conditions of \cite[Theorem 7.7]{Kajino2010}, cited in Theorem~\ref{thm:kajino}, which is needed for our  Theorem \ref{thm:funcequal}. These conditions are satisfied by the class of generalized Sierpinski carpets, as in \cite{BB,BB2,BB3,BBKT}, and for finitely ramified symmetric fractals. 
Naturally more detailed accounts of these conditions and examples of fractals which satisfy them can be found in \cite{Kajino2010}.

A Dirichlet form $(\C{E},\C{F})$ satisfies condition $(CHK)$ 
if it admits a jointly continuous heat kernel. It satisfies $(UHK)$  if the heat kernel satisfies a sub-Gaussian estimate with a metric that is adapted to the self-similar structure of \C{L}. 
\emph{Intersection type finite} is the property of $K$ that the intersection between cells at the same scale can be of only some finite collection of homomorphic types. For example, the standard Sierpinski gasket's self-similar structure (the cells of the same scale are the same as the cells at some particular depth in the construction which only intersect at most one point) or for the Sierpinski carpet (two cells of at the same scale intersect either along a whole edge or at a corner) are Intersection type finite. 
%
%Conditions $(LWTF)$ and  $(SSDFF3\C{S})$ are the local weight type finite and strong domain self-similarity conditions on the self-similar Dirichlet space $(\C{L},\mu,\C{E},\C{F},r)$. 
The local weight type finite $(LWTF)$ condition holds if the set of ratios between the energy scaling factors, $r_w$, of neighboring cells and the set of ratios of the measure scaling factors $\mu_{w}$ of neighboring cells are both finite. That is, the ratio between the energy and measure scalings between neighboring cells of any size can take only finitely many values. The strong domain self-similarity $(SSDF3\C{S})$ condition holds when the action of a $\C{L}-$isomorphism takes the set of localized continuous functions of finite energy to itself. A localized function is one whose support is not the entire set $F$.  

%%%

A subset $X \subset W_{\#}$  is \emph{independent} if there is an injection of $\Sigma(X)$ into $\Sigma(S)$. 
Furthermore, 
the subset $X$ is \emph{separated} if it is non-empty, finite, independent, and the number of $n$ such that $\sigma^{n}(w) \in \sigma_x \Sigma$ is 
 finite for all $w \in \Sigma$ for some  $x \in W_{\#}$. 
 Here $\sigma^{n}$ is the $n'th$ iterate of the shift map and $\sigma_x \Sigma$  are all words beginning with the prefix $x$. 
 Note the difference between the usage of the superscipt and subscript. The rational boundary $(RB)$ condition holds when there exists some integer $N \in \B{N}$ and a separated set $X_k \subset W_{\#}$ and $w_k \in W_*$ for each $k =1, \ldots , N$ such that the post-critical set associated to the iterated function system satisfies $\C{P}_{\C{L}} = \bigcup_{i=1}^{N} \Sigma_{wk}[X_k]$. To quote Kajino: ``[r]oughly speaking, $(RB)$ says that the boundary $V_0$ is a finite union of self-similar sets.'' The notion of rational boundary for a self-similar set was introduced in \cite{KiHKE}.

%%%
In the following, 
%$X$ is a non-empty finite subset of $W_{\#}$. T
the functions $Z_b$ are partition functions with various boundary conditions, where the value of $b$ indicates where Dirichlet conditions are imposed (in particular, $Z_D$ is the partition function where the Dirichlet conditions are imposed only at the boundary of \C{L}), 
%. When it appears 
and $d_{\partial}$ will be the cell-counting dimension of $V_0$ with respect to the scale~\C{S}.

\begin{theorem}{\cite[Theorem 7.7]{Kajino2010}}\label{thm:kajino}
Assume that $K$ is connected and that $(\C{E},\C{F})$ is conservative. Suppose that $(\C{L},\C{S})$ is {\bfseries intersection type finite} and that $(LWTF)$, $(SSDF3\C{S})$, $(CHK)$ and $(UHK)$ hold. Let $F \subset K$ be a closed subset of $K$, let $w \in W_*$ and let $X \subset W_{\#}$ be separated and satisfy $Cap_{\C{E}}(K[X]) >0$. Set $L := F \cup K_w[X]$ and $d_{\partial} :=d(\gamma, X)$ and suppose $F \subset L \subset K$. Then there exists $c_1,c_2 \in (0,\infty)$ such that for any $t \in (0,1]$,
\begin{equation*}
	c_1t^{-d_{\partial}/2} \le Z_{F^{c}}(t) - Z_{L^{c}}(t) \le c_2t^{-d_{\partial}/2}.
\end{equation*}
\end{theorem}

\begin{corollary}{\cite[Corollary 7.8]{Kajino2010}}\label{cor:kajino}
Assume that $K$ is connected and that $(\C{E},\C{F})$ is conservative. Suppose that $(\C{L},\C{S})$ is intersection type finite and that $(LWTF)$, $(SSDFF3\C{S})$, $(CHK)$ and $(UHK)$ hold. Suppose also that $\gamma_i = \gamma$ for any $i \in S$ for some $\gamma \in (0,1)$ and that $\C{L}$ satisfies $(RB)$ with $N \in \B{N}$ and $X_k \subset W_{\#}$ for $k \in \{1, \ldots , N\}$. Let $d_{\partial} :=\max_{1 \le k \le N} d(\gamma,X_k)\ $
%(=dim_{\C{S}} V_0 \in [0,d_{s})$ by Theorem 6.9) 
and let $G$ be the continuous $\log(\gamma^{-1})-$periodic function given in Corollary 5.3. If $Cap_{\C{E}}(K[X_J]) >$ for some $J \in \{1, \ldots , N\}$ satisfying $d(\gamma,X_j) = d_{\partial}$ then there exists $c_1,c_2 \in (0,\infty)$ such that for any $t \in (0,1]$,
\begin{equation*}
	c_1t^{-d_{\partial}/2} \le t^{-d_s/2}G\left(\frac{1}{2}\log\frac{1}{t}\right) - Z_D(t) \le c_2t^{-d_{\partial}/2}.
\end{equation*}
\end{corollary}

\subsection*{Acknowledgments } 
{The authors are very  grateful to Eric Akkermans, Gerald Dunne and Naotaka Kajino for  important  helpful discussions leading to this paper. Thanks also to Joe Chen for helping to verify some of the calculations and Luke Rogers for suggesting reference~\cite{CRSS2007}. {Research supported in part by the National Science Foundation, grants DMS-0652440 (first author) and DMS-0505622 (second author).}}

%\newpage\bibliographystyle{plain}
%\bibliography{merozeta}{}\end{document}
%

\

\end{document}